\documentclass{article}
\usepackage{amsmath,amssymb,amscd,amsthm,verbatim,alltt,amsfonts,array}
\usepackage[english]{babel}
\usepackage{latexsym}
\usepackage{amssymb}
\usepackage{euscript}
\usepackage{graphicx}

\newtheorem{theorem}{Theorem}
\theoremstyle{plain}

\newtheorem{corollary}{Corollary}

\newtheorem{definition}{Definition}

\newtheorem{lemma}{Lemma}

\newtheorem{remark}{Remark}

\numberwithin{equation}{section}

\begin{document}

\title{Montel's Theorem and subspaces of distributions which are $\Delta^m$-invariant }

\author{J. M. Almira}

\date{}

\markboth{J. M. Almira }{ Montel's Theorem and  subspaces of distributions which are $\Delta^m$-invariant}

\maketitle 

\begin{abstract}
We study the finite dimensional spaces $V$ which are invariant under the action of the finite differences operator $\Delta_h^m$. Concretely, we prove that if $V$ is such an space, there exists a finite dimensional translation invariant space $W$ such that $V\subseteq W$. In particular, all elements of $V$ are exponential polynomials. Furthermore, $V$ admits a decomposition $V=P\oplus E$ with $P$ a space of polynomials and $E$ a translation invariant space.  As a consequence of this study, we prove a generalization of a famous result by P. Montel  \cite{montel} which states that, if $f:\mathbb{R}\to \mathbb{C}$ is a continuous function satisfying $\Delta_{h_1}^mf(t) = \Delta_{h_2}^mf(t)=0$ for all $t\in\mathbb{R}$ and certain $h_1,h_2\in\mathbb{R}\setminus\{0\}$ such that $h_1/h_2\not\in\mathbb{Q}$, then $f(t)=a_0+a_1t+\cdots+a_{m-1}t^{m-1}$ for all $t\in\mathbb{R}$ and certain complex numbers $a_0,a_1,\cdots,a_{m-1}$. We demonstrate, with quite different arguments, the same result not only for ordinary functions $f(t)$ but also for complex valued distributions. Finally,  we also consider in this paper the subspaces $V$  which are $\Delta_{h_1h_2\cdots h_m}$-invariant for all $h_1,\cdots,h_m\in\mathbb{R}$. 
\end{abstract}

\begin{quotation}

   % Se\-con\-dary       xxyxx
\end{quotation}

% \noindent{\bf Running title}
 % $p$-adic Montel theorem and locally polynomial functions 

\thispagestyle{empty}

\section{Motivation}
Let $X$ denote either the space of continuous functions $f:\mathbb{R}\to \mathbb{C}$ or the space of complex valued Schwartz distributions. A well known result by P. M. Anselone and J. Korevaar \cite{anselone} characterizes the finite dimensional subspaces of $X$ which are translation invariant as the spaces of solutions of the homogeneous linear differential equations with constant coefficients $x^{(n)}+a_1x^{(n-1)}+\dots+a_{n-1}x'+a_nx=0$ (here $x:\mathbb{R}\to \mathbb{C}$ and $a_1,\cdots,a_n\in \mathbb{C}$ for some $n\in\mathbb{N}$) (see also \cite{leland}, where a simpler proof of this result is given for spaces of continuous functions). These spaces are generated by a set of monomials of the form 
\begin{equation}\label{base}
t^{k-1}e^{\lambda t}, \ k=1,\cdots,m(\lambda) \text{ and } \lambda\in \{\lambda_0,\lambda_1,\cdots,\lambda_s\}\subset \mathbb{C},
\end{equation} 
so that their elements are exponential polynomials. We assume, by convention, that $\lambda_0=0$ and that $m(\lambda_0)=0$ means that this set does not contain elements of the form $t^{k}$ with $k\in \mathbb{N}$.  Recall that a subspace $V$ of $X$ is translation invariant if for all $h\in\mathbb{R}$ we have that $\tau_{h}(V)\subseteq V$, where $\tau_h(f)(t)=f(t+h)$ if $f$ is an ordinary function and $\tau_{h}(f)\{\phi\}=f\{\tau_{-h}(\phi)\}$ if $f$ is a distribution and $\phi$ is a test function. In their paper \cite{anselone} Anselone and Korevaar proved that if $V$ is a finite dimensional subspace of $X$ and $\tau_{h_1}(V)\subseteq V$, $\tau_{h_2}(V)\subseteq V$ for certain non-zero real numbers $h_1,h_2$ such that  $h_1/h_2\not\in\mathbb{Q}$, then $V$ is translation invariant and hence it admits an algebraic basis of the form \eqref{base}.  They also proved that, if $V$ is a finite dimensional subspace of the space of continuous complex valued functions defined on the semi-infinite interval $(0,\infty)$ and $\tau_{h_k}(V)\subseteq V$  for an infinite sequence of positive real numbers $\{h_k\}_{k=1}^\infty$ which converges to zero, then $V$ admits an algebraic basis of the form \eqref{base}.  

It is evident that, if we denote by $1_d=\tau_0$ the identity operator, and we define the first difference operator $\Delta_h=\tau_h-1_d$, then $\tau_h(V)\subseteq V$ if and only if $\Delta_h(V)\subseteq V$ (we say that $V$ is $\Delta_h$-invariant), so that the results in \cite{anselone} can be directly stated for finite dimensional spaces invariant by the operators $\Delta_h$ .   The main goal of this paper is to study the finite dimensional spaces $V$ which are invariant under the action of the finite differences operators $\Delta_h^m$, which are defined inductively by $\Delta_h^1=\Delta_h$ and $\Delta_h^{k+1}f=\Delta_h(\Delta_h^kf)$, $k=1,2,\cdots$. Concretely, we prove that if $V$ is such an space, there exists a finite dimensional translation invariant space $W$ such that $V\subseteq W$. In particular, all elements of $V$ are exponential polynomials. Furthermore, $V$ admits a decomposition $V=P\oplus E$ with $P$ a space of polynomials and $E$ a translation invariant space.  As a consequence of this study, we prove a generalization of a famous result by P. Montel  \cite{montel} which states that, if $f:\mathbb{R}\to \mathbb{C}$ is a continuous function satisfying $\Delta_{h_1}^mf(t) = \Delta_{h_2}^mf(t)=0$ for all $t\in\mathbb{R}$ and certain $h_1,h_2\in\mathbb{R}\setminus\{0\}$ such that $h_1/h_2\not\in\mathbb{Q}$, then $f(t)=a_0+a_1t+\cdots+a_{m-1}t^{m-1}$ for all $t\in\mathbb{R}$ and certain complex numbers $a_0,a_1,\cdots,a_{m-1}$. We demonstrate, with quite different arguments, the same result not only for ordinary functions $f(t)$ but also for complex valued distributions.  We devote the last section of this paper  to consider the subspaces $V$ of $X$ which are $\Delta_{h_1h_2\cdots h_m}$-invariant for all $h_1,\cdots,h_m\in\mathbb{R}$. Here $\Delta_{h_1h_2\cdots h_m}$ denotes the finite differences operator defined inductively by $\Delta_{h_1h_2\cdots h_m}f(t)=\Delta_{h_1}(\Delta_{h_2\cdots h_m}f)(t)$. Obviously, these operators generalize the operators $\Delta_h^m$, which are got when we impose the restriction $h_1=\cdots=h_m=h$. 

\section{$\Delta^m$-invariant subspaces and Montel's Theorem for distributions}

\begin{theorem}\label{VdentroW} Assume that $V$ is a finite dimensional subspace of $X$ and $\Delta_{h_1}^m(V)\subseteq V$, $\Delta_{h_2}^m(V)\subseteq V$ for certain non-zero real numbers $h_1,h_2$ such that  $h_1/h_2\not\in\mathbb{Q}$. Then there exists a
finite dimensional subspace $W$ of  $X$ which is invariant by translations and contains $V$. Consequently, all elements of $V$ are exponential polynomials. \end{theorem}

\begin{lemma} \label{uno} Let $E$ be a vector space and $L:E\to E$ be a linear operator defined on $E$. If $V\subset E$ is an $L^m$-invariant subspace of $E$, then the space
\[
\Box_L^m(V)=V+L(V)+L^2(V)+\cdots+L^m(V)
\] 
is $L$-invariant. Furthermore, $\Box_L^m(V)$ is the smallest $L$-invariant subspace of $E$ containing $V$.
\end{lemma}

\begin{proof} This result is trivial. Its proof is an easy exercise.

%The linearity of $L$ implies that 
%\[
%L(\Box_L^m(V))=L(V)+L^2(V)+L^3(V)+\cdots+L^m(V)+L^{m+1}(V).
%\]
%Now, $L^{m+1}(V)=L(L^{m}(V))\subseteq L(V)$ and $L(V)+L(V)=L(V)$, so that $L(\Box_L^m(V))\subseteq \Box_L^m(V)$. 

%On the other hand, let us assume that $V\subseteq F\subseteq E$ and $F$ is an $L$-invariant subspace of $E$. If $\{v_k\}_{k=0}^m\subseteq V$, then 
%$L^k(v_k)\in F$ for all $k\in \{0,1,\cdots,m\}$, so that $v_0+L(v_1)+\cdots+L^m(v_m)\in F$. This proves that $\Box_L^m(V)\subseteq F$.
\end{proof}

\begin{lemma} \label{dos} Let $E$ be a vector space and $L,S:E\to E$ be two linear operators defined on $E$. Assume that $LS=SL$.  If $V\subset E$ is a vector subspace of $E$ which satisfies  $L^m(V)\cup S^m(V)\subseteq V$, then 
\[
S^m(\Box_L^m(V)) \subseteq \Box_L^m(V).
\] 
Consequently, the space
\[
\diamond_{L,S}^m(V)= \Box_S^m(\Box_L^m(V)) 
\]
is $L$-invariant, $S$-invariant, and contains $V$. 
\end{lemma}

\begin{proof}
By definition, 
\begin{eqnarray*}
S^m(\Box_L^m(V)) &=& S^m(V+L(V)+L^2(V)+\cdots+L^m(V))\\
&=& S^m(V)+L(S^m(V))+L^2(S^m(V))+\cdots+L^m(S^m(V))) \\
&\subseteq&  V+L(V)+L^2(V)+\cdots+L^m(V) = \Box_L^m(V),
\end{eqnarray*}
since  $S,L$  commute. This proves that $\Box_L^m(V)$ is $S^m$-invariant, and Lemma \ref{uno} implies that $\diamond_{L,S}^m(V)= \Box_S^m(\Box_L^m(V)) $ is $S$-invariant. On the other hand, the identity $SL=LS$ implies that
\begin{eqnarray*}
L(\diamond_{L,S}^m(V)) &=&  L(\Box_L^m(V)+S(\Box_L^m(V))+ \cdots+S^m(\Box_L^m(V)))\\
&=& L(\Box_L^m(V))+S(L(\Box_L^m(V)))+\cdots+S^m(L(\Box_L^m(V)))  \\
&\subseteq & \Box_L^m(V)+S(\Box_L^m(V))+\cdots+S^m(\Box_L^m(V)) = \diamond_{L,S}^m(V),
\end{eqnarray*} 
so  that $\diamond_{L,S}^m(V)$ is $L$-invariant. Finally, $V\subseteq \Box_L^m(V) \subseteq \diamond_{L,S}^m(V)$.
\end{proof}

\begin{proof}[Proof of Theorem \ref{VdentroW}] We apply Lemma \ref{dos} with $E=X$, $L=\Delta_{h_1}$ and $S=\Delta_{h_2}$ to conclude that 
$V\subseteq W=\diamond_{\Delta_{h_1},\Delta_{h_2}}^m(V)$ and $W$ is a finite dimensional subspace of $X$ satisfying $\Delta_{h_i}(W)\subseteq W$ , $i=1,2$. Hence we can apply Anselone-Korevaar's Theorem to $W$ and conclude that this space admits an algebraic basis of the form \eqref{base}. In particular, all elements of $V$ are exponential polynomials.
\end{proof}

\begin{remark} Note that the space $W$ we have constructed for the proof of Theorem \ref{VdentroW}, satisfies  $\dim_{\mathbb{C}}W\leq (m+1)^2\dim_{\mathbb{C}}V$. With a different proof, we can substitute this inequality by the exact formula $\dim_{\mathbb{C}}W = m\dim_{\mathbb{C}}V$ (see the proof of Theorem $\ref{semirecta}$ below, which can be fully adapted to this context). 
\end{remark}

It is interesting to observe that condition $h_1/h_2\not\in\mathbb{Q}$ in Theorem \ref{VdentroW} can not be weakened. Indeed, if  $h_1/h_2 \in\mathbb{Q}$ and $m\geq 1$ then there are finite dimensional subspaces $V$ of $X$ such that $\Delta_{h_1}^m(V)\cup \Delta_{h_2}^m(V) \subseteq V$ and no finite dimensional translation invariant subspace $W$ of $X$  satisfies $V\subseteq W$. To construct these spaces we need to use the following technical results:
\begin{lemma}\label{propiedaddiferencias} Assume that $f:\mathbb{R}\to \mathbb{C}$ satisfies $\Delta_h^mf=0$, and let $p\in \mathbb{Z}$. Then $\Delta_{ph}^mf=0$. 
\end{lemma}
\begin{proof} For $m=1$ the result is trivial, since the periods of the function $f$ form an additive subgroup of $\mathbb{R}$. For $m\geq 2$ the result follows from commutativity of the composition of the operators $\Delta_h$ (i.e., we use that $\Delta_h\Delta_kf=\Delta_k\Delta_hf$). Concretely, $\Delta_h^mf=\Delta_h(\Delta_h^{m-1}f)=0$ implies that $h$ is a period of $\Delta_h^{m-1}f$. Hence $ph$ is also a period of this function and $\Delta_{ph}(\Delta_h^{m-1}f)=0$. Now we use that $\Delta_{ph}(\Delta_h^{m-1}f)=\Delta_{h}(\Delta_{ph}\Delta_h^{m-2}f)$ and iterate the argument several times to conclude that $\Delta_{ph}^mf=0$.
\end{proof}
\begin{lemma} \label{ecuaciondiferencias} Let $h>0$ and assume that $g\in\mathbf{C}(\mathbb{R})$ satisfies $g(h\mathbb{Z})=\{0\}$. Then there exists $f\in\mathbf{C}(\mathbb{R})$ such that $f(h\mathbb{Z})=\{0\}$ and $\Delta_hf=g$. Consequently, for each $m\geq 1$ there exists $F_m\in\mathbf{C}(\mathbb{R})$ such that $F_m(h\mathbb{Z})=\{0\}$ and $\Delta_h^mF_m=g$. 
\end{lemma}

\begin{proof} Given $g\in\mathbf{C}(\mathbb{R})$ satisfying $g(h\mathbb{Z})=\{0\}$, it is easy to check that the function 
\[
f(z)=\left \{
\begin{array}{cccccc}
\sum_{j=0}^{k-1}g(x+jh) & \text{if} & z=x+kh, \ k\in\mathbb{N}\setminus\{0\}, \text{ and } x\in [0,h) \\
-\sum_{j=1}^{k}g(x-jh) & \text{if} & z=x-kh, \ k\in\mathbb{N}\setminus\{0\}, \text{ and } x\in [0,h) \\
0 & \text{if} & z\in [0,h) \\
\end{array} \right.
\]
is continuous and satisfies $f(h\mathbb{Z})=\{0\}$ and $\Delta_hf=g$. The second claim of the lemma follows by iteration of this argument.
\end{proof} 
Let us now assume that $h_1,h_2>0$, $h_1/h_2\in\mathbb{Q}$ , and $m\in\mathbb{N}$ ($m\geq 1$). Obviously, there exists $h>0$ and $p,q\in\mathbb{N}$ such that $h_1=ph$ and $h_2=qh$. Consider the function $\phi\in\mathbf{C}(\mathbb{R})$ defined by $\phi(x)=|x|$ for $|x|\leq h/2$ and $\phi(x+h)=\phi(x)$ for all $x\in\mathbb{R}$ and use Lemma \ref{ecuaciondiferencias} with $g=\phi$ to construct, for each $m\geq 2$, a function $f_m\in\mathbf{C}(\mathbb{R})$ such that  $\Delta_h^{m-1}f_m=\phi$. Take $f_1=\phi$. Then $\Delta_h^{m}f_m=\Delta_h\phi=0$ for all $m$. This, in conjunction with Lemma \ref{propiedaddiferencias}, implies that the one dimensional space $V_m=\mathbf{span}\{f_{m}\}$ satisfies   $\Delta_{h_1}^{m}(V_m)\cup \Delta_{h_2}^{m}(V_m)=\{0\} \subseteq V_m$. On the other hand, $V_m$ cannot be contained into any finite dimensional translation invariant subspace of $X$, since $f_m\in V_m$ is not an exponential polynomial (indeed, it is not an analytic function).

\begin{corollary}[Montel's Theorem for distributions] \label{monteltheorem} Assume that $f$ is a complex valued distribution such that $\Delta_{h_1}^mf =\Delta_{h_2}^mf=0$ for certain non-zero real numbers $h_1,h_2$ such that  $h_1/h_2\not\in\mathbb{Q}$. Then $f$ is an ordinary polynomial of degree $\leq m-1$.
\end{corollary}

\begin{proof} Assume that $\Delta_{h_1}^mf =\Delta_{h_2}^mf =0$. Then $V=\mathbf{span}\{f\}$ is a one dimensional space of complex valued distributions which   
satisfies the hypotheses of Theorem \ref{VdentroW}. Hence all elements of $V$ are exponential polynomials. In particular, $f$ is an exponential polynomial,
\[
f(t)=\sum_{k=0}^{m(\lambda_0)-1}a_{0,k}t^{k}+\sum_{i=1}^s \sum_{k=0}^{m(\lambda_i)-1}a_{i,k}t^{k}e^{\lambda_i t}
\]
and we can assume that $m(0)\geq m$ with no loss of generality. Let 
\[
\beta = \{
 t^{k-1}e^{\lambda_i t}, \ \  k=1,\cdots,m(\lambda_i) \text{ and } i=0,1,2,\cdots,s\} 
\]
and $\mathcal{S}=\mathbf{span}\{\beta\}$ be an space with a basis of the form \eqref{base} which contains $V$.  Let us consider the linear map $\Delta_{h}:\mathcal{S}\to\mathcal{S}$ induced by the operator $\Delta_h$ when restricted to $\mathcal{S}$.  The matrix associated to this operator with respect to the basis $\beta$ is block diagonal,  $A=\mathbf{diag}[A_0,A_1,\cdots,A_s]$, with 
\begin{equation}\label{A0}
A_0=\left [
\begin{array}{cccccc}
0 & h & h^2 & \cdots & h^{m(0)-1} \\
0 & 0 & 2h & \cdots & \binom{m(0)-1}{2}h^{m(0)-2}\\
\vdots & \vdots & \ddots & \cdots & \vdots \\
0 & 0 & 0 & \cdots &  \binom{m(0)-1}{m(0)-2}h\\
0 & 0 & 0 & \cdots &  0
\end{array} \right]
\end{equation}
and
\begin{equation} \label{Ai}
A_i=\left [
\begin{array}{cccccc}
e^{\lambda_ih}-1 & h e^{\lambda_ih}& h^2e^{\lambda_ih} & \cdots & h^{m(i)-1}e^{\lambda_ih} \\
0 & e^{\lambda_ih}-1  & 2h e^{\lambda_ih} & \cdots & \binom{m(i)-1}{2}h^{m(i)-2}e^{\lambda_ih}\\
\vdots & \vdots & \ddots & \cdots & \vdots \\
0 & 0 & 0 & \cdots &  \binom{m(i)-1}{m(i)-2}h e^{\lambda_ih}\\
0 &  0 & 0 & \cdots &  e^{\lambda_ih}-1
\end{array} \right],
\end{equation}
for $i=1,2,\dots, s$. It follows that the matrix associated to $(\Delta_h^m)_{|\mathcal{S}}$ with respect to the basis $\beta$ is given by $A^m= \mathbf{diag}[A_0^m,A_1^m,\cdots,A_s^m]$. Obviously, the matrices $A_i^m$ ($i=1,2,\cdots,s$) are invertible since the corresponding  $A_i$  are so. On the other hand, $\textbf{rank}(A_0^m)=m(0)-m$ and $$\mathbf{ker}(A_0^m)=\mathbf{span}\{(0,0,\cdots,0,1^{\text{(i-th position)}},0,\cdots, 0): i=1,2,\cdots,m\}.$$ It follows that 
$\mathbf{rank}(A^m)=\dim_{\mathbb{C}}\mathcal{S}-m$, so that $\dim_{\mathbb{C}} \mathbf{ker}(A^m)=m$. On the other hand, a simple computation shows that the space of ordinary polynomials of degree $\leq m-1$, which we denote by $\Pi_{m-1}$, is contained into 
$\mathbf{ker}(\Delta_h^m)$.  Hence $\mathbf{ker}(\Delta_h^m)=\Pi_{m-1}$, since both spaces have the same dimension. This, in conjunction with  $f\in \mathbf{ker}(\Delta_h^m)$, ends the proof. 

\end{proof}

\begin{theorem} \label{semirecta}  Let $V$ be a finite dimensional subspace of the space of continuous complex valued functions defined on the semi-infinite interval $(0,\infty)$ and assume that $\Delta_{h_k}^m(V)\subseteq V$  for an infinite sequence of positive real numbers $\{h_k\}_{k=1}^\infty$ which converges to zero. Then all elements of $V$ are exponential polynomials. 
\end{theorem} 

\begin{proof} Assume that $\dim_{\mathbb{C}}V=N<\infty$ and $\Delta_{h_k}^m(V)\subseteq V$  for an infinite sequence of positive real numbers $\{h_k\}_{k=1}^\infty\searrow 0$. Let $\{e_1(t),\cdots,e_N(t)\}$ be a basis of $V$ and let $h\in\{h_k\}_{k=1}^\infty$. Then 
\begin{equation} \label{1}
\Delta_h^me_i(t)=\sum_{k=1}^Na_{ik}(h)e_k(t)\  \text{ for all } t\in (0,\infty) \text{ and } i=1,\cdots,N.
\end{equation}
Let us set $E(t)=(e_1(t),e_2(t),\cdots,e_N(t))^T$ (where $v^T$ denotes the transpose of the vector $v$). Then \eqref{1} can be written in matrix form as
\begin{equation} \label{2}
\Delta_h^mE(t)=A(h)E(t)\  \text{ for all } t\in (0,\infty),
\end{equation}
where $A(h)=(a_{ik}(h))_{i,k=1}^N$ is a matrix function of $h$. 

We regularize the functions $e_i(t)$ via convolution with a test function. More precisely, we consider  $\varphi(t)$ an infinitely differentiable function with compact support and introduce the new functions $f_{i,\varphi}(t)=(e_i\ast\varphi)(t)$. These functions are of class $\mathbb{C}^{(\infty)}(0,\infty)$ and, if we define $F_{\varphi}(t)= (f_{1,\varphi}(t),f_{2,\varphi}(t),\cdots,f_{N,\varphi}(t))^T$, then 
\begin{equation} \label{3}
\Delta_h^mF_{\varphi}(t)=A(h)F_{\varphi}(t)\  \text{ for all } t\in (0,\infty),
\end{equation}
since the operation of convolution is translation invariant. The novelty here is the fact that $F_{\varphi}$ is an infinitely differentiable function. Taking $\{\varphi_n\}$ a sequence of test functions converging (in distributional sense)  to Dirac's delta function $\delta$, we have that $f_{i,\varphi_n}=e_i\ast\varphi_n\to e_i\ast\delta=e_i$, so that we can impose that, for a certain $n_0$, the set of fuctions $\{f_{i,\varphi_{n_0}}\}_{i=1}^N$ is linearly independent (since the functions $\{e_i(t)\}$ form a basis). This computation implies that we can assume, with no loss of generality, that our functions   $\{f_{i,\varphi}\}_{i=1}^N$ define a basis of the space they span. In particular, there exists a set of $N$ points $\{t_i\}_{i=1}^N\subset (0,\infty)$ such that $\det\mathbf{col}[F_{\varphi}(t_1),F_{\varphi}(t_2),\cdots,F_{\varphi}(t_N)]\neq 0$. It follows that
\[
A(h)=\mathbf{col}[\Delta_h^mF_{\varphi}(t_1),\cdots,\Delta_h^mF_{\varphi}(t_N)]\mathbf{col}[F_{\varphi}(t_1),\cdots,F_{\varphi}(t_N)]^{-1}
\]
and 
\[
\frac{A(h)}{h^m}=\mathbf{col}[\frac{\Delta_h^mF_{\varphi}(t_1)}{h^m},\cdots,\frac{\Delta_h^mF_{\varphi}(t_N)}{h^m}]\mathbf{col}[F_{\varphi}(t_1),\cdots,F_{\varphi}(t_N)]^{-1}.
\]
Taking limits for $h$ converging to zero, we get
\begin{eqnarray*}
\lim_{h\to 0}\frac{A(h)}{h^m} &=& \mathbf{col}[(F_{\varphi})^{(m)}(t_1),\cdots,(F_{\varphi})^{(m)}(t_N)]\mathbf{col}[F_{\varphi}(t_1),\cdots,F_{\varphi}(t_N)]^{-1}\\
&=& B\in M_N(\mathbb{C}).
\end{eqnarray*}
Furthermore, the convergence of $A(h)/h^m$ to the matrix $B$ is in the sense of all norms of $M_N(\mathbb{C})$, since this space is of finite dimension. This implies that, if we fix $K$ a compact subset of $(0,\infty)$ then
\[
\|\frac{A(h)}{h^m}E(t)-BE(t)\|_{\mathbf{C}(K)}\leq \|\frac{A(h)}{h^m}-B\|\|E(t)\|_{\mathbf{C}(K)}\to 0 \text{ (for } h\to 0\text{ )}.
\] 
It follows that $\frac{1}{h^m}\Delta_h^mE$ converges, in distributional sense, to $BE(t)$. On the other hand, a simple computation shows that, if $\phi$ is any test function, then 
\begin{eqnarray*}
\frac{1}{h^m}\Delta_h^mE\{\phi\} &=& E\{\frac{1}{h^m}\Delta_{-h}^m\phi\}\\
 &=& E\{(-1)^m\frac{1}{(-h)^m}\Delta_{-h}^m\phi\}\to E\{(-1)^m\phi^{(m)}\} =E^{(m)}\{\phi\}.
\end{eqnarray*}
Thus $E^{(m)}=BE$ in distributional sense. But the continuity of $E(t)$ implies that \begin{equation} \label{differential} E^{(m)}=BE\end{equation} in the ordinary sense. Let us set $\mathcal{E}=(E,E',\cdots,E^{(m-1)})^T$. This transforms  equation \eqref{differential} into the linear ordinary differential equation
\begin{equation} \label{D2}
\mathcal{E}'=\left[ \begin{array}{cccccc}
0 & I & 0 & 0 & \cdots & 0 \\
0 & 0 & I & 0&  \cdots & 0\\
\vdots & \vdots & \ddots & \cdots  & \cdots & \vdots \\
0 & 0 & 0 & \cdots & 0 & I\\
B & 0 & 0 & \cdots & 0 &  0
\end{array} \right] \mathcal{E}
\end{equation}
Now, it is well known that each component of any solution $\mathcal{E}(t)$ of \eqref{D2} is a finite linear combination of exponential monomials of the form  \eqref{base} for an appropriate choice of the values $\lambda$ and $m(\lambda)$.
\end{proof}

\section{Characterization of $\Delta^m$-invariant subspaces}
Let us state the main result of this section.
\begin{theorem} \label{main} Assume that $V$ is a finite dimensional subspace of $X$ which satisfies $\Delta_h^m(V)\subseteq V$ for all $h\in\mathbb{R}$. Then 
there exist vector spaces $P\subset \Pi:=\mathbb{C}[t]$ and $E\subset \mathbf{C}(\mathbb{R})$ such that $V=P\oplus E$ and $E$ is invariant by translations. Consequently, $V$ is invariant by translations if and only if $P$ is so. 
\end{theorem}
In order to prove this theorem, we first need to introduce some notation and results about invariant subspaces of linear maps in the finite dimensional context. The main reference for this subject is \cite{gohberg}.

\begin{definition} Given $T:\mathbb{C}^n\to\mathbb{C}^n$ a linear transformation, and $\lambda$ any of its eigenvalues, we define the root subspace -or generalized eigenspace- associated to $\lambda$ and $T$ by the formula $R_{\lambda}(T)=\mathbf{ker}(T-\lambda I)^n$, where $I:\mathbb{C}^n\to\mathbb{C}^n$ denotes the identity operator. 
%The algebraic multiplicity of  $\lambda$ is defined by $m(\lambda,T)=\dim_{\mathbb{C}}R_{\lambda}(T)$. The geometric multiplicity of $\lambda$ is  $g(\lambda,T)=
%\dim_{\mathbb{C}}\mathbf{ker}(T-\lambda I)$. 
\end{definition}

The next result is well known (see \cite[Theorem 2.1.5, page 50]{gohberg}):

\begin{theorem}\label{decomposition} Let $T:\mathbb{C}^n\to\mathbb{C}^n$ be a linear transformation and  $V$ be a linear subspace of $\mathbb{C}^n$. Let $\{\lambda_0,\cdots,\lambda_t\}$ be the set of all (pairwise distinct) eigenvalues of $T$. Then $V$ is $T$-invariant if and only if $V=(V\cap R_{\lambda_0}(T))\oplus \cdots \oplus (V\cap R_{\lambda_t}(T))$ and each subspace $V_i=(V\cap R_{\lambda_i}(T))$ is $T$-invariant. 
\end{theorem} 

%\begin{theorem} \label{marked} Let $T:\mathbb{C}^n\to\mathbb{C}^n$ be a linear transformation such that every eigenvalue $\lambda$ of $T$ satisfies either $(a)$ the %geometric multiplicity of $\lambda$ is equal to its algebraic multiplicity , or $(b)$ $\dim \mathbf{ker}(T-\lambda I)=1$. Then all $T$-invariant subspaces of $\mathbb{C}^n$ %are marked. 
%\end{theorem} 

Finally, the following lemma should be known to experts. We include the proof here for the sake of completeness, since this result forms an essential part of our arguments for the proof of Theorem \ref{main}.

\begin{lemma} \label{nuevo} Let $E$ be a vector space with basis  $\beta=\{v_k\}_{k=1}^n$ and let $m\in\mathbb{N}$, $m\geq 1$. Assume that $T:E\to E$ is such that $A=M_{\beta}(T)$ is of the form $A=\lambda I+B$, where $\lambda\neq 0$ and $B$ is strictly upper triangular with nonzero entries in the first superdiagonal. Then the full list of $T$-invariant subspaces of $E$ is given by $V_0=\{0\}$ and $V_k=\mathbf{span}\{v_1,\cdots,v_k\}$, $k=1,2,\cdots,n$. Furthermore $T^m$ has the same invariant subspaces as $T$. \end{lemma}

\begin{proof} Assume that $A=M_{\beta}(T)$ is of the form $A=\lambda I+B$, where $\lambda\neq 0$ and $B$ is strictly upper triangular with nonzero entries in the first superdiagonal, and let $V\neq \{0\}$ be a $T$-invariant subspace.  Let $v\in V$, $v=a_1v_1+\cdots+a_sv_s$, $a_s\neq 0$. Then $w=Tv-\lambda v\in V$ and a simple computation shows that $w=\alpha_1v_1+\cdots+\alpha_{s-1}v_{s-1}$ with $\alpha_{s-1}=b_{s-1,s}a_s\neq 0$, where $B=(b_{ij})_{i,j=1}^n$. It follows that, if $V$ is $T$-invariant and  $v=a_1v_1+\cdots+a_sv_s\in V$ with $a_s\neq 0$, then $\mathbf{span}\{v_1,v_2,\cdots,v_s\}\subseteq V$. Take $k_0=\max\{k:\text{ exists }v\in V, v=a_1v_1+\cdots+a_sv_s\text{ and }a_s\neq 0\}$. Then $V=\mathbf{span}\{v_1,\cdots,v_{k_0}\}$. Finally, it is clear that all the spaces $V_k=\mathbf{span}\{v_1,\cdots,v_k\}$, $k=1,2,\cdots,n$ are $T$-invariant.

To compute the invariant subspaces of $T^m$ we take into account that $A^m=M_{\beta}(T^m)$ and 
\begin{eqnarray*}
A^m &=& (\lambda I+B)^m\\
&=& \sum_{k=0}^m\binom{m}{k}(-1)^{m-k}B^k\\
&=& \lambda^mI+m\lambda^{m-1}B+\sum_{k=2}^m\binom{m}{k}(-1)^{m-k}B^k.
\end{eqnarray*}
This shows that $A^m=\lambda^mI+C$, with $C$ strictly upper triangular  with nonzero entries in the first superdiagonal, since the only contribution  to the first superdiagonal  of  $C=m\lambda^{m-1}B+\sum_{k=2}^m\binom{m}{k}(-1)^{m-k}B^k$ is got from $m\lambda^{m-1}B$, and $\lambda\neq 0$ . This proves that we can apply the first part of the lemma to the linear transformation $T^m$, which concludes the proof. 
\end{proof}

\begin{remark} It is important to note that there are many examples of linear transformations $T:E\to E$ such that $T$ and $T^m$ have different sets of invariant subspaces. For example, if $T$ is not of the form $T=\lambda I$ for any scalar $\lambda$ and satisfies $T^m=I$ or $T^m=0$, then all subspaces of $E$ are invariant subspaces of $T^m$  and, on the other hand, there exists $v\in E$ such that $Tv\not\in \mathbf{span}\{v\}$, so that $\mathbf{span}\{v\}$ is not an invariant subspace of $T$.  
\end{remark}

\begin{proof}[Proof of Theorem \ref{main}] It follows from Theorem \ref{VdentroW} that $V$ is a subspace of 
$\mathcal{S}=\mathbf{span}\{t^k\}_{k=0}^{m(\lambda_0)-1}\oplus \bigoplus_{i=1}^s \mathbf{span}\{t^ke^{\lambda_it}\}_{k=0}^{m(\lambda_i)-1}$ for certain values of $\lambda_i$, $m(\lambda_i)$, and $s$ (recall that we imposed $\lambda_0=0$). Furthermore, we have already computed the matrix $A=M_{\beta}(\Delta_h)$ associated to $(\Delta_h)_{|\mathcal{S}}$ with respect to the basis 
$\beta=\{t^k\}_{k=0}^{m(\lambda_0)-1} \cup \bigcup_{i=1}^s \{t^ke^{\lambda_it}\}_{k=0}^{m(\lambda_i)-1}$, which is given by $A=\mathbf{diag}[A_0,A_1,\cdots,A_s]$, with $A_0$ satisfying \eqref{A0} and $A_i$ satisfying \eqref{Ai}, $i=1,2,\cdots,s$. It follows that $A^m=\mathbf{diag}[A_0^m,A_1^m,\cdots,A_s^m]$ is the matrix associated to $(\Delta_h^m)_{|\mathcal{S}}$ with respect to $\beta$. 

In particular, the eigenvalues of  $(\Delta_h^m)_{|\mathcal{S}}$ are given by 
$\{0,(e^{\lambda_ih}-1)^m,i=1,2,\cdots,s\}$. A direct computation shows that $$R_{0}((\Delta_h^m)_{|\mathcal{S}})=\Pi_{m(\lambda_0)-1}=\mathbf{span}\{t^k\}_{k=0}^{m(\lambda_0)-1}$$ and $$R_{(e^{\lambda_ih}-1)^m}((\Delta_h^m)_{|\mathcal{S}})=\mathbf{span}\{t^ke^{\lambda_it}\}_{k=0}^{m(\lambda_i)-1}.$$
Hence Theorem \ref{decomposition} shows that $V$ is $\Delta_h^m$-invariant if and only if $V=V_0\oplus V_1\oplus\cdots \oplus V_s$, where $V_0\subseteq \Pi_{m(\lambda_0)-1}$ and $V_i$ is a $\Delta_h^m$-invariant subspace of   $E_i=R_{(e^{\lambda_ih}-1)^m}((\Delta_h^m)_{|\mathcal{S}})$, $i=1,\cdots,s$. Thus, to find all $\Delta_h^m$-invariant subspaces of $\mathcal{S}$ (one of them being our space $V$) we only need to consider the invariant subspaces of  the spaces $\Pi_{m(0)-1}$ and $E_i$, $i=1,2,\cdots,s$. Now, $\beta_i=\{t^ke^{\lambda_it}\}_{k=0}^{m(\lambda_i)-1}$ is a basis of $E_i$ and  $A_i=M_{\beta_i}((\Delta_h)_{|E_i})$ is given by \eqref{Ai}, so that we can apply Lemma  \ref{nuevo} to $(\Delta_h)_{|E_i}$ and conclude that $V_i\subset E_i$ is $\Delta_h^m$-invariant if and only if it is  $\Delta_h$-invariant. This proves the theorem with $P=V_0$ and $E= V_1\oplus\cdots \oplus V_s$. 
\end{proof}

\begin{remark}
There are many examples of spaces $P\subseteq \Pi$ which are finite dimensional, $\Delta^m$-invariant and non translation invariant. A typical example is $P=\mathbf{span}\{1,t^m\}$. Obviously, $\Delta_h^m(P)=\Pi_0=\mathbf{span}\{1\}\subseteq P$. On the other hand, if $h\neq 0$,  $\Delta_ht^m=\sum_{k=0}^{m-1}\binom{m}{k}h^{m-k}t^k\not\in P$.  
\end{remark}

\section{$\Delta_{h_1h_2\cdots h_m}$-invariant subspaces}
In this section we consider the subspaces $V$ of $X$ which are $\Delta_{h_1h_2\cdots h_m}$-invariant for all $h_1,\cdots,h_m\in\mathbb{R}$. Our first result characterizes the property of $\Delta_{h_1h_2\cdots h_m}$-invariance for arbitrary subspaces of $X$. This result is an easy consequence of a well known theorem by D. Z. Djokovi\'{c} \cite{Dj} (see also \cite[Theorem 7.5, page 160]{HIR}, \cite[Theorem 15.1.2., page 418]{kuczma}), which states that  the operators $\Delta_{h_1 h_2\cdots h_s}$ satisfy the equation
\begin{equation*} 
%\label{igualdad}
\Delta_{h_1\cdots h_s}f(t)=
\sum_{\epsilon_1,\dots,\epsilon_s=0}^1(-1)^{\epsilon_1+\cdots+\epsilon_s}
\Delta_{\alpha_{(\epsilon_1,\dots,\epsilon_s)}(h_1,\cdots,h_s)}^sf(t+\beta_{(\epsilon_1,\dots,\epsilon_s)}(h_1,\cdots,h_s)),
\end{equation*}
where $$\alpha_{(\epsilon_1,\dots,\epsilon_s)}(h_1,\cdots,h_s)=(-1)\sum_{r=1}^s\frac{\epsilon_rh_r}{r}$$ and $$\beta_{(\epsilon_1,\dots,\epsilon_s)}(h_1,\cdots,h_s)=\sum_{r=1}^s\epsilon_rh_r.$$  
Later on (see Theorem \ref{main2}), we prove that if $V$ is a finite dimensional subspace of $X$ which is $\Delta^m_h$-invariant for all $h\in\mathbb{R}$ then $V$ is  $\Delta_{h_1h_2\cdots h_m}$-invariant for all $h_1,h_2,\cdots,h_m\in\mathbb{R}$.

\begin{theorem} \label{otro} Assume that $V$ is a subspace of $X$ which satisfies $\Delta_h^m(V)\subseteq V$ for all $h\in\mathbb{R}$. Then the following statements are equivalent: 
\begin{itemize}
\item[$(i)$] $\Delta^{m-1}_h\Delta_s(V)\subseteq V$ for all $h,s\in\mathbb{R}$.
\item[$(ii)$] $\Delta_{h_1h_2\cdots h_m}(V)\subseteq V$ for all $(h_1,h_2,\cdots,h_m)\in\mathbb{R}^m$. 
\end{itemize}
\end{theorem}

\begin{proof} We  only prove $(i)\Rightarrow (ii)$ since the other implication is trivial. Now, Djokovi\'{c}'s Theorem implies that the operator 
$\Delta_{h_1h_2\cdots h_{m-1}}$ is a linear combination of operators of the form $\Delta_h^{m-1}\Delta_s$ (with $h,s\in\mathbb{R}$). Hence $(i)$ implies that 
$\Delta_{h_1h_2\cdots h_{m-1}}(V)\subseteq V$, for all $(h_1,h_2,\cdots,h_{m-1})\in\mathbb{R}^{m-1}$. On the other hand, the identity $\Delta_{h_1h_2}=\Delta_{h_1+h_2}-\Delta_{h_1}-\Delta_{h_2}$ (which is easy to check) implies that, if $(h_1,h_2,\cdots,h_{m})\in\mathbb{R}^{m}$, then
\[
\Delta_{h_1h_2\cdots h_m}=\Delta_{h_1h_2}\Delta_{h_3\cdots h_m}=\Delta_{(h_1+h_2)h_3\cdots h_m}-\Delta_{h_1h_3\cdots h_m}-\Delta_{h_2h_3\cdots h_m}
\]
and $V$ is invariant by all operators appearing in the last member of the identity above. 
\end{proof}
\begin{theorem} \label{main2} Assume that $V$ is a finite dimensional subspace of $X$. Then the following statements are equivalent:
\begin{itemize}
\item[$(i)$]  $\Delta_h^m(V)\subseteq V$ for all $h\in\mathbb{R}$. 
\item[$(ii)$] $\Delta_{h_1h_2\cdots h_m}(V)\subseteq V$  for all $(h_1,h_2,\cdots,h_m)\in\mathbb{R}^m$. 
\end{itemize}
\end{theorem}

\begin{proof}
We only prove $(i)\Rightarrow (ii)$ since the other implication is trivial. It follows from Theorem \ref{main} that $V=P\oplus E$ with $E$ translation invariant and $P$ a subspace of $\Pi$. Thus we only need to prove that $\Delta_{h_1h_2\cdots h_m}(P)\subseteq P$, for all $(h_1,h_2,\cdots,h_m)\in\mathbb{R}^m$. Take $N=\min\{n:P\subseteq \Pi_n\}$ and let $A_m(h)=M_{\beta}((\Delta_h^m)_{|\Pi_N})$ be the matrix associated to $(\Delta_h^m)_{|\Pi_N}:\Pi_N\to\Pi_N$ with respect to the natural basis $\beta=\{1,t,t^2,\cdots,t^N\}$. This matrix can be computed explicitly just taking into account the following well known formulas:
\begin{equation}
\Delta_h^mf(t)=\sum_{k=0}^m\binom{m}{k}(-1)^{m-k}f(t+kh)
\end{equation}  
(see, e.g., \cite[Corollary 15.1.2, page 418]{kuczma}) and
\begin{equation}
\sum_{k=0}^m\binom{m}{k}(-1)^kk^r=(-1)^m \left \{
\begin{array}{cccccc} r \\ m \end{array}\right\}
m!,
\end{equation}
where $\left \{ \begin{array}{cccccc} r \\ m \end{array}\right\}$ denote the Stirling numbers  of the second kind (see, e.g. \cite[Identity 18, page 3136]{spivey}). Recall that these numbers satisfy the recurrence relations
\begin{eqnarray*}
\left \{
\begin{array}{cccccc} n+1 \\ k \end{array}\right\} = k\left \{
\begin{array}{cccccc} n \\ k \end{array}\right\} + \left \{
\begin{array}{cccccc} rn\\ k-1 \end{array}\right\} \\
\left \{
\begin{array}{cccccc} 0 \\ 0 \end{array}\right\} = 1,\ \ 
\left \{
\begin{array}{cccccc} n \\ 0 \end{array}\right\} = \left \{
\begin{array}{cccccc} 0\\ n \end{array}\right\} = 0 
\end{eqnarray*}
and $\left \{
\begin{array}{cccccc} n \\ k \end{array}\right\}=0$ if $n<k$. 
Now, given $s\in\{0,1,\cdots,N\}$ we have that 
\begin{eqnarray*}
\Delta_h^mt^s &=& \sum_{k=0}^m\binom{m}{k}(-1)^{m-k}(t+kh)^s\\
&=& \sum_{k=0}^m\left[\binom{m}{k}(-1)^{m-k}\left( \sum_{j=0}^s \binom{s}{j}(kh)^{s-j}t^j\right)\right] \\
&=& \sum_{j=0}^s \left[\binom{s}{j}\left(\sum_{k=0}^m\binom{m}{k}(-1)^{m-k}(kh)^{s-j}\right)t^j \right]\\
&=& \sum_{j=0}^s \left[\binom{s}{j}(-1)^m\left(\sum_{k=0}^m\binom{m}{k}(-1)^{k}(k)^{s-j}\right)h^{s-j}t^j \right]\\
&=& \sum_{j=0}^s \left[\binom{s}{j}(-1)^m\left((-1)^m \left \{
\begin{array}{cccccc} s-j \\ m \end{array}\right\}
m!\right)h^{s-j}t^j \right]\\
&=& \sum_{j=0}^s \binom{s}{j} \left \{
\begin{array}{cccccc} s-j \\ m \end{array}\right\}
m! h^{s-j}t^j .
\end{eqnarray*}
It follows that $A_m(h)=(a_{ij}(h))_{i,j=0}^N$ is given by 
\[
a_{ij}(h)= \left\{\begin{array}{cccccc}   \left(\begin{array}{cccccc}  j \\ i \end{array}\right) \left \{
\begin{array}{cccccc} j-i \\ m \end{array}\right\}
m! h^{j-i} & & \text{ for } i=0,1,\cdots,j \text{ and } j=m,\cdots,N \\
0 & & \text{otherwise}
\end{array}\right. 
\]
Let $p(t)=b_0+b_1t+\cdots+b_Nt^N\in P$ be a polynomial with $b_N\neq 0$. Then the coordinates of $\Delta_h^mp(t)$ with respect to the basis $\beta$ are given by 
\[
A_m(h) \left [ \begin{array}{cccccc} 
b_0 \\ b_1\\ \vdots \\b_N \end{array}\right]= \left [ \begin{array}{cccccc} 
f_0(h) \\ f_1(h)\\ \vdots \\  f_{N-m}(h)\\ 0 \\ \vdots \\ 0\end{array}\right],
\]
where $f_i(h)=\sum_{j=m}^Na_{ij}(h)b_j= \sum_{j=m+i}^N \left(\begin{array}{cccccc}  j \\ i \end{array}\right) \left \{
\begin{array}{cccccc} j-i \\ m \end{array}\right\}
m! h^{j-i} b_j$ is a polynomial (in the variable $h$) satisfying $\deg f_i(h)=N-i$, $i=0,1,\cdots,N-m$, since $b_N\neq 0$. This proves that the functions 
$\{f_0(h),\cdots,f_{N-m}(h)\}$ form a linearly independent set and, as a consequence, there exists numbers $\{h_i\}_{i=0}^{N-m}$ such that  the vectors $v_i=[f_0(h_i),\cdots,f_{N-m}(h_i)]^T$, 
$i=0,1,\cdots,N-m$,  are such that the matrix 
$H=\mathbf{col}[v_0,\cdots,v_{N-m}]$ is invertible. In particular, these vectors form a basis of $\mathbb{C}^{N-m+1}$. This proves that $\Pi_{N-m}\subseteq P$ since $\Delta_h^mp(t)\in P$ for all $h\in\mathbb{R}$. It follows that  
\[
\Delta_{h_1h_2\cdots h_m} P\subseteq \Pi_{N-m} \subseteq P.
\]
for all $h_1,h_2,\cdots,h_m\in\mathbb{R}$, since $P\subseteq \Pi_N$ and $\Delta_{h_1h_2\cdots h_m} \Pi_N = \Pi_{N-m} $.
\end{proof} 

\section{The case of real valued distributions}
Let $Z$ denote either the real vector space of continuous functions $f:\mathbb{R}\to \mathbb{R}$ or the space of real valued Schwartz distributions. The results we have demonstrated in the previous sections of  this paper can be stated, with appropriate modifications, for subspaces of $Z$. Indeed, the real valued results are a direct consequence of the theorems we have already proved and the use of complexi\-fi\-cation, which is a standard tool in Linear Algebra and Functional Analysis. Concretely, given $V\subset Z$ a vector subspace of the real vector space $Z$, we define its complexification $V^{\mathbb{C}}$ as the vector subspace of $X$ defined by $V^{\mathbb{C}}=V+iV$. It is evident that 
\[
\Delta_h^m(V^{\mathbb{C}})=(\Delta_h^m(V))^{\mathbb{C}},
\] 
so that $V$ is a $\Delta_h^m$-invariant subspace of $Z$ if and only if $V^{\mathbb{C}}$ is a $\Delta_h^m$-invariant subspace of $X$. Let us consider separately the cases $m=1$ and $m>1$.

\noindent \textbf{Case $m=1$:}  Assume that $V$ is a finite dimensional subspace of $Z$ and $\Delta_h(V)\subseteq V$. Then $V^{\mathbb{C}}=V+iV$ is a translation invariant subspace of $X$, so that it admits, when considered as a complex vector space, a basis of the form
\[
\beta = \{
 t^{k-1}e^{\lambda_i t}, \ \  k=1,\cdots,m(\lambda_i) \text{ and } i=0,1,2,\cdots,s\}. 
\]
This obviously implies that $V$ admits, as a real vector space, a basis of the form 
\[
\gamma = \{
 t^{k-1}\}_{k=1}^{m_0} \cup \bigcup_{i=1}^s\bigcup_{k=1}^{m(\lambda_i)}\{t^{k-1}e^{\mathbf{Re}(\lambda_i) t}\cos(\mathbf{Im}(\lambda_i)t),t^{k-1}e^{\mathbf{Re}(\lambda_i) t}\sin(\mathbf{Im}(\lambda_i)t)\}. 
\]
\noindent\textbf{Case $m>1$:} If $V$ is a finite dimensional subspace of $Z$ and $\Delta_h^m(V)\subseteq V$. Then $V^{\mathbb{C}}=V+iV$ is a $\Delta_h^m$-invariant subspace of $X$, so that it admits a decomposition $V^{\mathbb{C}}=P+E$ with $P,E$ finite dimensional subspaces of $X$, $P$ being a $\Delta_h^m$-invariant subspace of $\Pi$ and $E$ a translation invariant subspace of $X$.  Taking real parts, we obtain that $V=F+G$ with $F$ a $\Delta_h^m$-invariant finite dimensional subspace of $\mathbb{R}[t]$ and $G$ a translation invariant finite dimensional subspace of $Z$.

Finally, we would like to comment that both Montel's Theorem (i.e., our Corollary \ref{monteltheorem}) and Theorems \ref{otro}, \ref{main2} are also true -with no changes- for the real case.

\noindent\textbf{Acknowledgement}\textit{.} I would like to express my warmest gratitude to the anonymous referee of this paper since his many clever and generous ideas have helped to write a much more better manuscript. Not only his remarks helped with the readability of the paper but also to simplify some proofs and to include some new interesting ideas. Because of all this: thanks.

\smallskip\noindent 
Received: 2013, May

\medskip 
 
\noindent Departamento de Matem\'{a}ticas \\Universidad de Ja\'{e}n \\ 
E-mail: jmalmira@ujaen.es  

\medskip

\end{document}